\documentclass[10pt,twoside]{article}

\usepackage[english]{babel}
\usepackage[utf8]{inputenc}
\usepackage{amsmath}
\usepackage{amsfonts}
\usepackage{mathtools}
\usepackage{amssymb}
\usepackage{amsthm}
\usepackage{hyperref}
\usepackage{graphicx}
\usepackage{epstopdf}
\usepackage{todonotes}

\newtheorem{Theorem}{Theorem}[section]
\newtheorem{Lemma}[Theorem]{Lemma}
\newtheorem{Corollary}[Theorem]{Corollary}
\newtheorem{Proposition}[Theorem]{Proposition}

\theoremstyle{definition}
\newtheorem{Definition}[Theorem]{Definition}
\newtheorem{Remark}[Theorem]{Remark}
\newtheorem{Example}[Theorem]{Example}

\begin{document}


\pagestyle{myheadings}


\title{On the convergence of generalized kernel-based interpolation by greedy data selection algorithms}

\date{\today}

\author{\large
Kristof Albrecht\footnote{Universit\"at Hamburg, Dept of Mathematics, {\tt kristof.albrecht@studium.uni-hamburg.de}} \qquad
Armin Iske\footnote{Universit\"at Hamburg, Department of Mathematics, {\tt armin.iske@uni-hamburg.de}}
}

\markboth{\footnotesize \rm \hfill K.~ALBRECHT AND A.~ISKE \hfill}
{\footnotesize \rm \hfill CONVERGENCE OF GENERALIZED KERNEL-BASED INTERPOLATION \hfill}

\maketitle
\thispagestyle{plain}


\begin{abstract}
We analyze the convergence of generalized kernel-based interpolation me\-thods.
This is done under minimalistic assumptions on both the kernel and the target function.
On these grounds, we further prove convergence of popular greedy data selection algorithms for totally bounded sets of sampling functionals.
Supporting numerical results concerning computerized tomography are provided for illustration.
\end{abstract}

\section{Introduction}
Kernel functions provide powerful numerical methods for multivariate scattered data approximation~\cite{Wendland2005}.
Positive definite or conditionally positive definite radial kernels --- previously also referred to as {\em radial basis functions} (RBF)~\cite{Buhmann2003} --- 
were primarily used to interpolate multivariate scalar-valued functions from Lagrange data. In that particu\-lar case, finitely many scalar 
samples are taken from a (multivariate) target function by the application of Dirac point evaluation functionals on shifted scattered points.
A more comprehensive discussion on multivariate Lagrange interpolation by (conditionally) positive definite functions can be found 
in~\cite{Buhmann2003,Wendland2005}.

Possible extensions of kernel-based {\em Lagrange} interpolation 
are working with more general evaluation functionals, rather than just Dirac point evaluations.
One classical example are evaluations of (directional) derivatives, whose samples are called {\em Hermite data}.
Other extensions include {\em Radon data} (i.e., from line integrals, cf.~\cite{DeMarchi2018}) or 
{\em cell averages} (e.g., in finite volume methods, cf.~\cite{Aboiyar2010}).

As shown in~\cite{Iske1995Hermite,Wu1992Hermite}, 
the concept of kernel-based interpolation from samples taken by (linear) functionals 
can be unified to obtain the more general Hermite-Birkhoff interpolation scheme,
which is throughout this paper called {\em generalized kernel-based interpolation}, cf.~\cite[Chapter~16]{Wendland2005}.

As regards the theory of kernel-based interpolation, we remark that positive defi\-nite kernels generate a
{\em reproducing kernel Hilbert space} (RKHS), i.e., the kernel's {\em native space} (cf.~\cite[Chapter~10]{Wendland2005}).
For the error analysis of kernel-based interpolation, the pointwise error functional (lying in the dual space of the RKHS) 
plays a central role. In fact, available results on the convergence and approximation orders of kernel-based interpolation
involve the norm of the pointwise error functional, called {\em power function} (cf.~\cite[Chapter~10]{Wendland2005}).

One important ingredient for the development of stable kernel-based interpolation algorithms are {\em Newton bases} (cf.~\cite{Mueller2009,Pazouki2011}).
Moreover, the implementation of {\em greedy algorithms} (cf.~\cite{Wenzel2023standard}) for the selection of suitable sample points
plays an important role for the performance of kernel-based interpolation schemes, especially for their stability and efficiency. 
We remark that the above mentioned ingredients (i.e., power function,  Newton bases, greedy algorithms) 
can be transferred to generalized kernel-based interpolation in the kernel's RKHS.

To make only one example, greedy algorithms were recently used in~\cite{Schaback_greedy2019,wenzel2022adaptive} to solve elliptic PDEs by meshfree kernel methods.
In the particular problem of~\cite{wenzel2022adaptive}, the utilized kernel is required to be sufficiently smooth, which then led to convergence rates
in the kernel's native Sobolev space. 

It is important to note that there are relevant applications where it is not desirable at all to work with {\em smooth} kernels.
One relevant example are WENO schemes for the numerical solution of hyperbolic conservation laws~\cite{Iske1996},
another one is concerning applications of computerized tomography~\cite{DeMarchi2018}, where less restrictive assumptions
on the regularity of the target are essentially required.

This gives rise to allow for {\em rougher} (target) functions.
To this end, Narcowich, Ward \& Wendland~\cite{Narcowich2006} developed error estimates for target functions outside the kernel's native space.
To be more precise, Sobolev-type error estimates and convergence rates were proven in~\cite{Narcowich2006} for Lagrange interpolation on compact domains,
where the kernel is required to be translation invariant and, moreover, the kernel needs to have a Fourier transform with algebraic decay around infinity.

In this paper, we prove convergence for generalized kernel-based interpolation schemes under very {\em mild} restrictions on both the kernel and the target function.
In particular, we do not require the kernel to be translation invariant. Moreover, we do not assume Sobolev regularity for the kernel or for the target function.

The outline of this paper is as follows. 
In Section~\ref{sec:native}, we introduce key ingredients of generalized kernel-based interpolation,
including the native space and the power function.
In Section~\ref{sec:convergence}, we first prove a convergence criterion (Theorem~\ref{power_convergence}), 
involving the asymptotic behaviour of the power function, on which we then prove convergence of generalized kernel-based interpolation (Theorem~\ref{filldistance_convergence})
under rather mild assumptions.
In Section 4, we use Theorem \ref{filldistance_convergence} to prove convergence for two popular classes of greedy algorithms:
the $\beta$-greedy algorithm (Theorem~\ref{beta_greedy}) and the geometric greedy algorithm of~\cite{DeMarchi2005} (Theorem~\ref{hgreedy_convergence}). 
Supporting numerical results on kernel-based interpolation from scattered Radon data are presented in Section~\ref{sec:numerics}. 

\section{Generalized kernel-based interpolation}
\label{sec:native}
We consider a positive definite function $K: \mathbb{R}^d \times \mathbb{R}^d \longrightarrow \mathbb{R}$, i.e., 
$K$ is symmetric and for any finite set $X = \lbrace x_1, \ldots , x_n \rbrace \subset \mathbb{R}^d$ of pairwise distinct points, 
the matrix
\begin{align*}
  A_{K,X} := \left( K(x_i,x_j) \right)_{1 \leq i,j \leq n} \in \mathbb{R}^{n \times n}
\end{align*}
is symmetric positive definite (see~\cite{Iske2018Approx,Wendland2005}). 
Recall that that any positive definite function $K$ generates a reproducing kernel Hilbert space (RKHS) $\mathcal{H}_K$, so that the reproduction
\begin{equation}
\label{eq:repro}
      f(x) = \langle f, K(\cdot, x) \rangle_{\mathcal{H}_K} \qquad \text{ for all } f \in \mathcal{H}_K \text{ and all } x \in \mathbb{R}^d
\end{equation} 
holds.
Therefore, $K$ is referred to as a {\em kernel function}. Moreover, $\mathcal{H}_K$ is called the kernel's {\em native space}.
An essential part of generalized interpolation in $\mathcal{H}_K$ is that the reproduction in~(\ref{eq:repro}) 
can be transferred to the dual space $\mathcal{H}_K^*$ of $\mathcal{H}_K$.

\begin{Proposition}[{\cite[Theorem 16.7]{Wendland2005}}]
\label{prop:repro}
  Let $\lambda \in \mathcal{H}_K^*$. For the function
  \begin{align*}
    \lambda^y K(\cdot,y) : \mathbb{R}^d \longrightarrow \mathbb{R}, 
    \qquad x \longmapsto \lambda^y K(x,y),
  \end{align*}
  we have $\lambda^y K(\cdot,y) \in \mathcal{H}_K$. Moreover, the generalized reproduction holds, i.e.,
  \begin{equation}
  \label{eq:general:repro}
    \lambda(f) = \langle f, \lambda^y K(\cdot,y) \rangle_{\mathcal{H}_K} \qquad \text{ for all } f \in \mathcal{H}_K.
  \end{equation}
\end{Proposition}

Note that Proposition~\ref{prop:repro} allows us to compute the Riesz representers for functionals from the dual space $\mathcal{H}_K^*$ explicitely,
which gives an advantage over the general Hilbert space setting. 

Now suppose that 
$\Lambda = \lbrace \lambda_1 , ... , \lambda_n \rbrace \subset \mathcal{H}_K^*$ is a finite set of pairwise distinct linear functionals. 
Given a function $f \in \mathcal{H}_K$ and scalar samples $\lambda_1(f), \ldots , \lambda_n(f)$ of $f$ on $\Lambda$, 
the generalized interpolation problem requires finding a function $s \equiv s_{f,\Lambda}$ from the finite dimensional linear subspace
\begin{align*}
  {\mathcal S}_{K,\Lambda} := \text{span} \; \lbrace \lambda_1^y K(\cdot,y), ... , \lambda_n^y K(\cdot,y) \rbrace \subset \mathcal{H}_K,
\end{align*}
such that the interpolation conditions
\begin{equation}
\label{eq:int:conds}
  \lambda_i(f) = \lambda_i(s) \qquad \text{ for all } i=1, \ldots , n
\end{equation}
are satisfied. 
If the linear functionals in $\Lambda$ are linearly independent, 
then there is a unique solution $s$ to the interpolation problem~(\ref{eq:int:conds}), cf.~\cite[Theorem 16.1]{Wendland2005}).
Thereby, we can define an interpolation operator
\begin{align*}
  I_{K,\Lambda}: \mathcal{H}_K \longrightarrow {\mathcal S}_{K,\Lambda}, \qquad f \longmapsto s_{f,\Lambda}.
\end{align*}
Due to the generalized reproduction in~(\ref{eq:general:repro}), 
the interpolant $s \equiv s_{f,\Lambda}$ coincides with the orthogonal projection 
of $f$ onto ${\mathcal S}_{K,\Lambda}$.

In our convergence analysis for generalized kernel-based interpolation, 
we consider, for a fixed subset of functionals $\Gamma \subset \mathcal{H}_K^*$, 
nested sequences $\left( \Lambda_n \right)_{n \in \mathbb{N}}$ 
of finite subsets $\Lambda_n \subset \Gamma$.
Without loss of generality, we assume throughout this paper that the functionals in $\Lambda_n$ are linearly independent, for every $n \in {\mathbb N}$. 
Otherwise, we can remove redundant basis elements from $\Lambda_n$.

Our general aim is to develop {\em mild} conditions on sequences 
$\left( \Lambda_n \right)_{n \in \mathbb{N}}$ under which the convergence, with respect to the $K$-norm $\|\cdot\|_K := \langle \cdot, \cdot \rangle^{1/2}_{\mathcal{H}_K}$,
\begin{align} \label{convergence_property} 
  \Vert f - I_{K,\Lambda_n} (f) \Vert_K \xrightarrow{n \to \infty} 0
\end{align}
holds for all target functions $f \in \mathcal{H}_{K,\Gamma}$ from the closed linear subspace
\begin{align*}
  \mathcal{H}_{K,\Gamma} = \overline{\text{span} \; \lbrace \lambda^y K(\cdot,y) \mid \lambda \in \Gamma \rbrace} \subset \mathcal{H}_K.
\end{align*}
We recall that the sequence $(I_{K,\Lambda_n} (f))_{n \in \mathbb{N}}$ of interpolants convergences in $\mathcal{H}_{K,\Gamma}$.
For further details on this, we refer to the second statement in Lemma~\ref{weakconv_sufficient}.

In the theory of standard kernel-based Lagrange interpolation, the pointwise error functional 
is an element in the dual $\mathcal{H}_K^*$ of the native space $\mathcal{H}_K$, whose 
$K$-norm $\|\cdot\|_K$, called {\em power function}, leads to an upper bound for pointwise error estimates.

It is rather straightforward to extend the power function of kernel-based Lagrange interpolation to the more general setting of this section:
For a given set of functionals $\Lambda \subset \Gamma \subset \mathcal{H}_K^*$, the {\em generalized power function} is defined as
\begin{align*}
  P_{\Lambda} (\lambda) := \Vert \lambda^y K(\cdot, y) - I_{K,\Lambda} (\lambda^y K(\cdot, y)) \Vert_{K}
  \qquad \text{ for } \lambda \in \Gamma.
\end{align*}
Note that the power function $P_{\Lambda}$ measures the $\|\cdot\|_{K}$-norm distance in ${\mathcal H}_K$
between the Riesz representer $\lambda^y K(\cdot,y)$ and the linear space ${\mathcal S}_{K,\Lambda}$
by the orthogonal projection property of the interpolation operator $I_{K,\Lambda}$.
As a direct consequence of the representation theorem, Proposition~\ref{prop:repro}, we have for $\lambda \in \Gamma$ the equivalence
\begin{align*}
  P_{\Lambda} (\lambda) = 0 \quad \Longleftrightarrow \quad \lambda^y K(\cdot, y) \in {\mathcal S}_{K,\Lambda} \quad \Longleftrightarrow \quad \lambda \in \text{span} \; (\Lambda).
\end{align*}
Hence, the power function can be interpreted as a measurement of the linear dependence between functionals in $\Gamma$. 
It can also be shown that the power function $P_\Lambda$ yields the upper bound
\begin{align*}
  \vert \lambda (f) - \lambda(I_{K,\Lambda}(f)) \vert \leq P_{\Lambda} (\lambda) \cdot \Vert f \Vert_{K} 
  \qquad \text{ for } f \in \mathcal{H}_K
\end{align*}
for the {\em pointwise} interpolation error. 
Another error indicator is the \textit{fill distance},
which is standard for error estimates of kernel-based Lagrange interpolation.
To comply with the more general setting of this section, we extend the notion of fill distances.

\begin{Definition}
\label{defi:fill_distance}
Let $(Z,d_Z)$ be a metric space and $A \subset Z$. 
We define the {\tt generalized fill distance} $h_{A,Z}$ of $A$ in $Z$ 
as the supremum of all distances to the subset $A$, i.e.,
\begin{align*}
      h_{A,Z} = \sup_{z \in Z} \, \inf_{a \in A} d_Z (z,a).
\end{align*}
\end{Definition}

In the setting of this section, the {\em generalized fill distance} is given as
\begin{equation}
\label{eq:generalized_fill_distance}
  h_{\Lambda, \Gamma} = \underset{\gamma \in \Gamma}{\sup} \; \underset{i=1,\ldots,n}{\min} \; \Vert \gamma - \lambda_i \Vert_{K},
\end{equation}
according to Definition~\ref{defi:fill_distance} (with $(Z,d_Z) = (\Gamma, \|\cdot\|_{K}$ and $A = \Lambda$).
In order to give a geometric description for the fill distance, note that $h_{\Lambda, \Gamma}$ 
in \eqref{eq:generalized_fill_distance}
is the radius of the largest open ball in the normed linear space $(\mathcal{H}_K^*,\|\cdot\|_K)$
with midpoint in $\Gamma$ that does not contain any point from $\Lambda$. 
To further explain on this, let us finally make a few relevant examples for scattered data interpolation from different data types.

\begin{Example} (Lagrange data)
For fixed domain $\Omega \subset {\mathbb R}^d$, let $\Gamma \equiv \Gamma(\Omega) = \{ \delta_x \, {\bf : } \, x \in \Omega \} \subset \mathcal{H}^*_K$
denote the set of all Dirac point evaluation functionals with support in $\Omega$. Moreover, for a fixed finite set $X = \{ x_1, \ldots, x_n\} \subset \Omega$ of 
pairwise distinct points in $\Omega$, let $\Lambda \equiv \Lambda(X) = \{ \delta_x \, {\bf : } \, x \in X \} \subset \Gamma$ be the set of point evaluations at the points in $X$.
In this special case, the fill distance in \eqref{eq:generalized_fill_distance} is
$$
    h_{\Lambda, \Gamma} = \underset{y \in \Omega}{\sup} \; \underset{i=1,...,n}{\min} \; \Vert \delta_y - \delta_{x_i} \Vert_{K}.
$$
Later, in Section~\ref{sub:parameterization}, we explain how to use parameterizations of functionals in order to simplify computations for the fill distance.
If we work with the standard parameterization $\varrho: \mathbb{R}^d \longrightarrow \mathcal{H}_K^*$ for point evaluations, where $x \longmapsto \delta_x$, 
cf.~\eqref{eq:parameterization}, then the fill distance of Lagrange interpolation boils down to the commonly used
$$
    h_{X,\Omega} = \underset{y \in \Omega}{\sup} \; \underset{i=1,...,n}{\min} \; \Vert y - x_i \Vert_2.
$$
\end{Example}

\begin{Example} (Radon data)
For radius $r \in \mathbb{R}$ and angle $\theta \in \left[ 0, \pi \right)$, let
\begin{equation}
\label{radon_functional}
    {\mathcal R}_{r,\theta} (f) := \int\limits_{{\mathbb R}} f( r \cdot \cos(\theta) - s \cdot \sin(\theta) , r \cdot \sin(\theta) + s \cdot \cos(\theta) ) \; ds
\end{equation}
denote the {\em Radon transform} of an integrable function $f$ at $(r,\theta) \in {\mathbb R} \times [0,\pi)$. 
We remark that \eqref{radon_functional} gives the integral of $f$ over the straight line 
$$
    \ell \equiv \ell(r,\theta) = \{ (r \cdot \cos(\theta) - s \cdot \sin(\theta) , r \cdot \sin(\theta) + s \cdot \cos(\theta)) \, {\bf : } \, s \in {\mathbb R} \} \subset {\mathbb R}^2.
$$
Later in Section~\ref{sec:numerics}, we discuss kernel-based reconstruction from Radon data, 
where $\Gamma = \{ {\mathcal R}_{r,\theta} \, {\bf : } \, (r,\theta) \in {\mathbb R} \times [0,\pi) \}$ is the set of all line integral operators of the form~\eqref{radon_functional}.
Moreover, $\Lambda = \{\lambda_i \, {\bf : } \, i=1,\ldots,n \} \subset \Gamma$ contains the $n$ line integral operators $\lambda_i \equiv {\mathcal R}_{r_i,\theta_i}$,
$$
    f \longmapsto {\mathcal R}_{r_i,\theta_i}(f)
    \qquad \mbox{ for } \ell_i \equiv \ell(r_i,\theta_i).
$$ 
Further details on kernel-based interpolation from scattered Radon data are explained in Section~5.
For a comprehensive account to Radon transforms and their applications to computerized tomography and nondestructive evaluation, we refer to~\cite{Natterer2001}.
\end{Example}

\begin{Example} (Cell averages)
In the numerical solution of hyperbolic conservation laws, finite volume schemes rely on a partitioning 
of a fixed computational domain $\Omega \subset {\mathbb R}^d$ into finitely many control volumes $V \subset \Omega$.
In this particular application, the unknown solution $u \equiv u(t,\cdot)$ of the hyperbolic PDE is, at time $t$, computed from 
{\em cell average values} over the control volumes,
\begin{equation}
\label{fvm_functional}
    {\mathcal A}_V (u (t, \cdot)) = \frac{1}{{\rm vol}_d (V)} \int_{V} u(t,x) \, dx.
\end{equation}
In this particular case, the set $\Gamma = \{ {\mathcal A}_V \, {\bf : } \, V \subset \Omega \}$ 
contains all integral operators of the form \eqref{fvm_functional} for control volumes $V$ in $\Omega$.
Moreover, for a stencil of control volumes ${\mathcal V} = \{V_1,\ldots,V_n\}$ (i.e., for finitely many control volumes from the partitioning of $\Omega$),
the set of interpolation functionals is $\Lambda = \{ {\mathcal A}_{V_i} \, {\bf : } \, i=1,\ldots,n \}$.
More details on kernel-based finite volume methods are explained in~\cite{Aboiyar2010}.
A more general account to finite volume methods for hyperbolic problems and their applications can be found in~\cite{LeVeque2002}.
\end{Example}

\begin{Remark}
 So far, it may not be clear how we can compute the power function and the interpolants in numerical tests. 
 As was already pointed out in \cite{wenzel2022adaptive}, the Newton basis from~\cite{Mueller2009} and 
 its respective update formula from \cite{Pazouki2011} can be adapted to the case of generalized interpolation. 
 Hence, (update) formulas for the power function and the interpolants are also available. 
\end{Remark}

\section{Convergence under mild conditions} 
\label{sec:convergence}
If we take a closer look at the desired convergence (\ref{convergence_property}), 
we see that this already includes the condition
\begin{align*}
    P_{\Lambda_n} (\lambda) = \Vert \lambda^y K(\cdot, y) - I_{K,\Lambda} (\lambda^y K(\cdot, y)) \Vert_{K} \xrightarrow{n \to \infty} 0 \qquad \text{for all } \lambda \in \Gamma,
\end{align*}
by definition of the subspace $\mathcal{H}_{K,\Gamma} \subset \mathcal{H}_{K}$.
Therefore, the pointwise decay of the power functions is a {\em necessary} condition. 
We can show that it is even {\em sufficient} for the normwise convergence of the interpolation method on the whole space.
Our following theorem generalizes the result~\cite[Proposition~2.1]{Karvonen2022} by Karvonen,
who covers the special case of kernel-based Lagrange interpolation.

\begin{Theorem} \label{power_convergence}
    Let $\Gamma \subset \mathcal{H}_K^*$ and $\left( \Lambda_n \right)_{n \in \mathbb{N}}$ be a nested sequence of finite 
    subsets of $\Gamma$. Then, the generalized interpolation method converges in $\mathcal{H}_{K,\Gamma}$, i.e.,
    \begin{align*}
        \Vert f - I_{K,\Lambda_n}(f) \Vert_{K} \xrightarrow{n \to \infty} 0 \qquad \text{for all } f \in \mathcal{H}_{K,\Gamma},
    \end{align*}
    if and only if the power functions converges pointwise to zero on $\Gamma$, i.e.
    \begin{align*}
        P_{\Lambda_n} (\lambda) \xrightarrow{n \to \infty} 0 \qquad \text{for all } \lambda \in \Gamma. 
    \end{align*}
\end{Theorem}

\begin{proof}
    For functionals $\mu_1,\ldots,\mu_N \in {\mathcal{H}^*_K}$ and coefficients $c_1,\ldots,c_N \in {\mathbb R}$,
    we consider
    \begin{align*}
        s = \sum\limits_{i=1}^{N} c_i \mu_i^y K(\cdot,y) \in \text{span}_{\mathbb{R}} \; \lbrace \lambda^y K(\cdot , y) \mid \lambda \in \Gamma \rbrace =: {\mathcal S}_{K,\Gamma}.  
    \end{align*}
    Since $I_{K,\Lambda_n}$ is a linear operator for each $n \in \mathbb{N}$, we get the convergence
    \begin{align*}
        \Vert s - I_{K,\Lambda_n}(s) \Vert_{K} &= \Big\Vert \sum\limits_{i=1}^N c_i \cdot \left( \mu_i^y K(\cdot,y) - I_{K, \Lambda_n} (\mu_i^y K(\cdot,y)) \right) \Big\Vert_{K} \\
        &\leq \sum\limits_{i=1}^N \vert c_i \vert \cdot P_{\Lambda_n} (\mu_i) \xrightarrow{n \to \infty} 0.
    \end{align*}

For any $f \in \mathcal{H}_{K,\Gamma}$ and $\varepsilon > 0$, 
there exists $s \in {\mathcal S}_{K,\Gamma}$ satisfying $\Vert f - s \Vert_{K} < \varepsilon / 3$,
due the density of ${\mathcal S}_{K,\Gamma}$ in $\mathcal{H}_{K,\Gamma}$. 
By the convergence $I_{K,\Lambda_n}(s) \longrightarrow s$, for $n \to \infty$, as already shown above, 
we get $\Vert s - I_{K,\Lambda_n}(s) \Vert_{K} < \varepsilon / 3$, for large enough $n$, so that
      \begin{align*}
        \Vert f - I_{K,\Lambda_n}(f) \Vert_{K} &\leq \Vert f - s \Vert_{K} + \Vert s - I_{K,\Lambda_n}(s) \Vert_{K} + \Vert I_{K,\Lambda_n}(s) - I_{K,\Lambda_n}(f) \Vert_{K} \\
        &\leq \Vert f - s \Vert_{K} + \Vert s - I_{K,\Lambda_n}(s) \Vert_{K} + \Vert I_{K,\Lambda_n} \Vert_{K} \cdot \Vert f - s \Vert_{K} \\
        &< \varepsilon
      \end{align*}
for large enough $n$, where we can use $\Vert I_{K,\Lambda_n} \Vert_{K} = 1$, since 
$I_{K,\Lambda_n} : {\mathcal H}_{K,\Gamma} \longrightarrow {\mathcal S}_{K,\Gamma}$ is an orthogonal projection operator.
\end{proof}

Similarly, we can derive a convergence condition with respect to the fill distances $h_{\Lambda_n, \Gamma}$. It should be mentioned here that this only represents a sufficient condition, and the decay of the fill distances might not be necessary for the convergence of the interpolation method.

\begin{Theorem} \label{filldistance_convergence}
    Let $\Gamma \subset \mathcal{H}_K^*$ and $\left( \Lambda_n \right)_{n \in \mathbb{N}}$ be a nested sequence of finite 
    subsets of $\Gamma$ satisfying $h_{\Lambda_n, \Gamma} \searrow 0$ for $n \to \infty$. Then we have
    \begin{align*}
      \Vert f - I_{K,\Lambda_n} (f) \Vert_{K} \xrightarrow{n \to \infty} 0 \qquad \text{for all } f \in \mathcal{H}_{K,\Gamma}.
    \end{align*}
  \end{Theorem}
  
  \begin{proof}
    It is easy to see that the fill distance is a uniform bound for the power function, so that we immediately get
    \begin{align*}
      P_{\Lambda_n} (\lambda) \leq h_{\Lambda_n, \Gamma} \xrightarrow{n \to \infty} 0 \qquad \text{ for all } \lambda \in \Gamma.
    \end{align*}
    Therefore, Theorem \ref{power_convergence} yields the stated convergence.
  \end{proof}

A key feature of native RKHS is that the norm convergence automatically implies pointwise convergence to the target function $f$. If the chosen kernel $K$ is bounded, we even have uniform convergence for the interpolation method.

\begin{Corollary}
  In the setting of Theorem \ref{power_convergence} or Theorem \ref{filldistance_convergence}, we also have uniform convergence
  \begin{align*}
    \Vert f - I_{K,\Lambda_n} (f) \Vert_{\infty} = \underset{x \in \mathbb{R}^d}{\sup} \; \vert f(x) - I_{K,\Lambda_n}(f)(x) \vert \xrightarrow{n \to \infty} 0,
  \end{align*}
  if $K$ is bounded.
\end{Corollary}

\begin{proof}
    This follows from our previous results and the standard estimate
    \begin{align*}
        \vert f(x) - I_{K,\Lambda_n}(f)(x) \vert &= \vert \langle f - I_{K,\Lambda_n} (f) , K(\cdot,x) \rangle_{\mathcal{H}_K} \vert \\
        &\leq \underset{x \in \mathbb{R}^d}{\sup} \; \sqrt{K(x,x)} \cdot \Vert f - I_{K,\Lambda_n} (f) \Vert_{K},
    \end{align*}
    as we have $\Vert K(\cdot,x) \Vert_{K} = \sqrt{K(x,x)}$ for all $x \in \mathbb{R}^d$.
\end{proof}

\subsection{Parameterization}
\label{sub:parameterization}
In many cases, the superset $\Gamma \subset \mathcal{H}_K^*$ is parameterized by a much simpler space. For example, the point evaluation functionals are parameterized by the mapping
\begin{equation}
\label{eq:parameterization}
    \varrho: \mathbb{R}^d \longrightarrow \mathcal{H}_K^*, \qquad x \longmapsto \delta_x.
\end{equation}
One main advantage of parameterizations is that, in general, many computations like evaluations of norms are less costly than in the native dual space. Instead of finding suitable sets of data points in the dual space, we can search for good data points in the parameter space. If the parameterization map $\varrho$ is uniformly continuous, the decay of the fill distances $h_{X_n,\Omega}$ in the parameter space then translates to the decay of the fill distances $h_{\varrho(X_n),\Gamma}$ in the dual space.

\begin{Theorem} \label{parameter_convergence}
    Let $\Gamma \subset \mathcal{H}_K^*$.
     Moreover, let $\left( \Omega, d_\Omega \right)$ be a metric space and $\varrho: \Omega \longrightarrow \Gamma$ be uniformly continuous and bijective. 
     If $\left( X_n \right)_{n \in \mathbb{N}}$ is a nested sequence of finite subsets from $\Omega$ satisfying $h_{X_n,\Omega} = \sup_{y \in \Omega} \min_{x \in X_n} d_\Omega(y,x) \searrow 0$ for $n \to \infty$, then we have the convergence
    \begin{align*}
        \Vert f - I_{K,\varrho(X_n)}(f) \Vert_{K} \xrightarrow{n \to \infty} 0 \qquad \text{for all } f \in \mathcal{H}_{K,\Gamma}.
    \end{align*}
\end{Theorem}

\begin{proof}
    According to Theorem \ref{filldistance_convergence}, it is sufficient to show $h_{\varrho(X_n),\Gamma} \searrow 0$ for $n \to \infty$. So let $\varepsilon > 0$. Since $\varrho$ is uniformly continuous, there is $\delta > 0$, such that the implication
    \begin{align*}
        d_{\Omega} (x,y) < \delta \quad \Longrightarrow \quad \Vert \varrho(x) - \varrho(y) \Vert_{K} < \varepsilon
    \end{align*}
    holds for all $x,y \in \Omega$. Moreover, we can find $N \in \mathbb{N}$ with $h_{X_N , \Omega} < \delta$. For any $\gamma \in \Gamma$, there is $x_\gamma \in \Omega$ satisfying $\varrho(x_\gamma) = \gamma$. Now choose $x \in X_N$ satisfying
    \begin{align*}
        d_{\Omega} (x_\gamma, x) = \underset{y \in X_N}{\min} \; d_{\Omega} (x_\gamma,y) < \delta.
    \end{align*}
    Then, we get
    \begin{align*}
        \underset{\lambda \in \varrho(X_N)}{\min} \; \Vert \gamma - \lambda \Vert_{K} \leq \Vert \varrho(x_\gamma) - \varrho(x) \Vert_{K} < \varepsilon.
    \end{align*}
    Altogether, this yields
    \begin{align*}
        h_{\varrho(X_n),\Gamma} \leq h_{\varrho(X_N), \Gamma} = \underset{\gamma \in \Gamma}{\sup} \; \underset{\lambda \in \varrho(X_N)}{\min} \; \Vert \gamma - \lambda \Vert_{K} \leq \varepsilon \qquad \text{ for } n \geq N.
    \end{align*}
\end{proof}

\section{Greedy data selection algorithms}
\label{sec:greedy}
The main concept of greedy algorithms is to find the optimal data point in terms of a pre-defined error indicator in each iteration. Given the current data set $\Lambda_n$ in the $n$-th step and an error function $\eta: \Gamma \longrightarrow \mathbb{R}$, the new functional is chosen via the selection rule
\begin{align*}
  \eta (\lambda_{n+1}) = \underset{\lambda \in \Gamma}{\sup} \; \eta(\lambda).
\end{align*}
Here, the selection rule certainly depends on the current data set $\Lambda_n$ (target-indepen\-dent), and it might also depend on the function $f$ that we want to approximate (target-dependent).

\begin{Remark}
  In most cases, the error functional $\eta$ is bounded and depends continuously on the input argument $\lambda$. But this does not guarantee that $\eta$ attains a maximum on $\Gamma$, as $\Gamma$ does not have to be compact in general. This problem can be fixed (theoretically) by introducing a parameter $\alpha \in \left( 0,1 \right)$ and changing the selection rule to a the weaker version (cf.~\cite{DeVore2013})
  \begin{align*}
    \eta (\lambda_{n+1}) \geq \alpha \cdot \underset{\lambda \in \Gamma}{\sup} \; \eta(\lambda).
  \end{align*}
  However, this adjustment does not affect our convergence results, so that we can ignore this consideration in our analysis.
\end{Remark}

Regarding the target-independent algorithms, we are interested in generating sequences of finite subsets satisfying the convergence conditions from Section~\ref{sec:convergence}. 
Recall that the required decay of the fill distances, i.e., $h_{\Lambda_n, \Gamma} \searrow 0$ for $n \to \infty$, gives a geometric condition. 
We remark that this decay condition on $(h_{\Lambda_n, \Gamma})_{n \in \mathbb{N}}$ can only hold, iff $\Gamma$ is \textit{totally bounded}. 
Recall that a set $\Gamma \subset \mathcal{H}_K^*$ is said to be \textit{totally bounded} in the normed linear space $(\mathcal{H}_K^*, \| \cdot \|_{K})$, 
iff for every $\varepsilon>0$ there are finitely many open balls $B(\gamma,\varepsilon) \subset \mathcal{H}_K^*$ with centers $\gamma \in \Gamma$ and radius $\varepsilon$, 
such that $\Gamma$ is covered by the union of these balls.
We will rely on this basic assumption for $\Gamma$ in our subsequent convergence analysis. 
Before we go into further details, we first collect a few useful auxiliary statements.

\begin{Lemma} \label{totalbounded_convergence}
  Let $\left( Z, d_Z \right)$ be totally bounded and $\left( a_i \right)_{i \in \mathbb{N}}$ be a sequence in $Z$. 
  Moreover, let $A_n := \lbrace a_1, ... , a_n \rbrace$, for $n \in \mathbb{N}$.
  Then we have the convergence
  \begin{align*}
    \textnormal{dist} (a_{n+1}, A_n) := \underset{i = 1, ... , n}{\min} \; d_Z(a_{n+1}, a_i) \xrightarrow{n \to \infty} 0.
  \end{align*} 
  In the case, where $Z = \Gamma \subset \mathcal{H}_K^*$ and $A_n = \Lambda_n = \lbrace \lambda_1, ... , \lambda_n \rbrace$,
   we have the convergence
  \begin{align*}
    P_{\Lambda_n} (\lambda_{n+1}) \xrightarrow{n \to \infty} 0.
  \end{align*}
\end{Lemma}

\begin{proof}
  Given $\varepsilon > 0$, we can find $z_1, ... , z_M \in Z$ satisfying
  \begin{align*}
    Z \subset \bigcup\limits_{j=1}^{M} B_{\varepsilon / 2} (z_j).
  \end{align*}
  Now we let
  \begin{align*}
    J_j :=
    \lbrace i \in \mathbb{N} \mid a_i \in B_{\varepsilon / 2} (z_j) \rbrace \qquad \text{for } j = 1 , ... , M
  \end{align*}
  and
  \begin{align*}
    N_j := 
  \left\{
    \begin{array}{cl}
      \min (J_j) & \mbox{ for } J_j \neq \emptyset, \\[1.5ex]
      0 & \mbox{ for } J_j = \emptyset
    \end{array}
\right.    
    \qquad \text{for } j = 1 , ... , M.
  \end{align*}
  Note that for $N := \max \lbrace N_j \mid j = 1, ... , M \rbrace$ and $n > N$, 
  there is one index $1 \leq j_n \leq M$ and one element $a \in A_N $ satisfying
  \begin{align*}
    \{ a_n , a \} \subset 
    B_{\varepsilon / 2} (z_{j_n}),
  \end{align*}
  due to the \textit{pigeonhole principle}. Using the triangle inequality, this implies
  \begin{align*}
    \text{dist}(a_n, A_N) \leq d_Z (a_n, a) \leq d_Z(a_n, z_{j_n}) + d_Z( z_{j_n}, a) < \varepsilon,
  \end{align*}
  whereby
  \begin{align*}
    \text{dist} (a_n, A_{n-1}) \leq \text{dist} (a_n, A_{N}) < \varepsilon.
  \end{align*}
  This completes our proof for the first statement.
  Finally, the second statement follows immediately from $P_{\Lambda_n}(\lambda_{n+1}) \leq \text{dist} (\lambda_{n+1}, \Lambda_n)$, for $n \in \mathbb{N}$.
\end{proof}

\begin{Lemma} \label{weakconv_sufficient}
For $\Gamma \subset \mathcal{H}_K^*$ let $f,g \in \mathcal{H}_{K,\Gamma}$.
Then, we have the equivalence
  \begin{align*}
    f = g \qquad \Longleftrightarrow \qquad \lambda(f) = \lambda(g) \quad \text{for all } \lambda \in \Gamma.
  \end{align*}
  Moreover, if $\left( \Lambda_n \right)_{n \in \mathbb{N}}$ is a nested sequence of finite subsets of $\Gamma$ satisfying
  \begin{align*}
    \lambda(I_{K,\Lambda_n}(f)) \xrightarrow{n \to \infty} \lambda(f) \qquad \mbox{ for all } \lambda \in \Gamma,
  \end{align*}
 for fixed $f \in \mathcal{H}_{K,\Gamma}$, then we have normwise convergence
  \begin{align*}
    \Vert f - I_{K,\Lambda_n}(f) \Vert_{K} \xrightarrow{n \to \infty} 0.
  \end{align*}
\end{Lemma}

\begin{proof}
  For the first statement, note that
  \begin{align*}
    \lambda(f) = \lambda(g) \quad \text{for all } \lambda \in \Gamma \qquad \Longleftrightarrow \qquad  f - g \perp \text{span}_{\mathbb{R}} \; \lbrace \lambda^y K(\cdot, y) \mid \lambda \in \Gamma \rbrace
  \end{align*}
  holds for $f,g \in \mathcal{H}_{K,\Gamma}$ due to the generalized reproduction property. 
  With the continuity of the inner product, this is equivalent to $f-g \perp \mathcal{H}_{K,\Gamma}$ and $f=g$. 
  
  For the second statement, regard for fixed $f \in \mathcal{H}_{K,\Lambda}$ the sequence
  \begin{align*}
    \left( \Vert I_{K,\Lambda_n} (f) \Vert_K^2 \right)_{n \in \mathbb{N}}.
  \end{align*}
  Due to the properties
  \begin{align*}
    I_{K,\Lambda_n}(I_{K,\Lambda_m}(f)) = I_{K,\Lambda_n}(f) \quad \text{for } m > n \qquad \text{and} \qquad \Vert I_{K,\Lambda_n}(f) \Vert_K \leq \Vert f \Vert_K
  \end{align*}
  for $ n \in \mathbb{N}$, this sequence is bounded and monotonically increasing. 
  Hence, it is convergent, and thereby a Cauchy sequence. From the identity
  \begin{align*}
    \Vert I_{K,\Lambda_m}(f) - I_{K,\Lambda_n}(f) \Vert_K^2 = \Vert I_{K,\Lambda_m}(f) \Vert_K^2 - \Vert I_{K,\Lambda_n}(f) \Vert_K^2,
  \end{align*}
  we can conclude that $\left( I_{K,\Lambda_n}(f) \right)_{n \in \mathbb{N}}$ is a Cauchy sequence, too. Since $\mathcal{H}_{K,\Gamma}$ is complete, there must be a normwise limit $g \in \mathcal{H}_{K,\Gamma}$. In this case, $g$ is also the weak limit of this sequence, so that we get
  \begin{align*}
    \lambda(g) = \lim\limits_{n \to \infty} \lambda (I_{K,\Lambda_n}(f)) = \lambda(f) \qquad \text{for all } \lambda \in \Gamma,
  \end{align*}
  due to our assumptions. Now the assertion follows from the first statement.
\end{proof}

\subsection{$\beta$-greedy algorithms}
The first class of greedy algorithms we analyze is the collection of $\beta$-greedy algorithms, which were introduced in \cite{Wenzel2023standard} and summarized many well-known greedy algorithms. For given $\beta \in \left[ 0, \infty \right)$, the selection rule is defined as
\begin{align*}
  \vert \lambda_{n+1} (f) - \lambda_{n+1} (I_{K,\Lambda_n}(f)) \vert^\beta \cdot P_{\Lambda_{n}} (\lambda_{n+1})^{1 - \beta} \\
  = \underset{\lambda \in \Gamma}{\sup} \; \vert \lambda (f) - \lambda (I_{K,\Lambda_n}(f)) \vert^\beta \cdot P_{\Lambda_{n}} (\lambda)^{1 - \beta}
\end{align*}
The most common choices for $\beta$ are the following: 
\begin{itemize}
  \item For $\beta = 0$, the resulting algorithm is known as the $P$-greedy algorithm (cf.~\cite{DeMarchi2005}) and aims to maximize the power function in each step. This is particularly improving the stability of the reconstruction methods, since the power function value acts as a lower bound for the smallest eigenvalue of the interpolation matrix (cf.~\cite[Theorem 12.1]{Wendland2005}), and is therefore critical for the condition of the problem.
  \item The choice $\beta = 1$ leads to the $f$-greedy algorithm (cf.~\cite{Schaback2000}) that maximizes the pointwise error. Hence, we aim to improve the interpolation error in areas where the current error is rather high.
  \item In order to achieve a good tradeoff between stability and approximation quality, it is reasonable to choose $\beta = 1/2$, which coincides with the psr-greedy method (cf.~\cite{Dutta2021}). 
  \item The approach can also be extended to the limit case $\beta \to \infty$, which corresponds to the $f/P$-greedy algorithm (cf.~\cite{Mueller2009Thesis,Schaback2006}). Here, we minimize the native space distance between the target function and the interpolant in each iteration.
\end{itemize}

In the case of totally bounded supersets, we can prove the convergence of the $\beta$-greedy algorithms for every $\beta \in \left[ 0, \infty \right]$.

\begin{Theorem} \label{beta_greedy}
  Let $\beta \in \left[ 0, \infty \right]$ and $\Gamma \subset \mathcal{H}_K^*$ be totally bounded. If $f \in \mathcal{H}_{K,\Gamma}$ and $\left( \Lambda_n \right)_{n \in \mathbb{N}}$ is chosen via the $\beta$-greedy algorithm, then we have the convergence
  \begin{align*}
    \Vert f - I_{K,\Lambda_n}(f) \Vert_{K} \xrightarrow{n \to \infty} 0.
  \end{align*}
\end{Theorem}

\begin{proof}
  According to Lemma \ref{weakconv_sufficient}, it is sufficient to show pointwise convergence. Therefore, let $\lambda \in \Gamma$. If $\lambda \in \bigcup_{n \in \mathbb{N}} \text{span} (\Lambda_n)$ holds, there is $n_0 \in \mathbb{N}$ such that $P_{\Lambda_n}(\lambda) = 0$ for $n \geq n_0$. The standard power function estimate then gives
  \begin{align*}
    \vert \lambda (f) - \lambda(I_{K,\Lambda_n} (f)) \vert \leq P_{\Lambda_n} (\lambda) \cdot \Vert f \Vert_K = 0 \qquad \text{for } n \geq n_0.
  \end{align*}
  
  Hence, let us assume that $\lambda \notin \bigcup_{n \in \mathbb{N}} \text{span} (\Lambda_n)$. We distinguish between several cases regarding the parameter $\beta$:

  \begin{itemize}
    \item \underline{$\beta = 0$:} In this case, we have
    \begin{align*}
      P_{\Lambda_n}(\lambda) \leq P_{\Lambda_n}(\lambda_{n+1}) \xrightarrow{n \to \infty} 0
    \end{align*}
    due to the selection rule and Lemma \ref{totalbounded_convergence}. Again, the standard power function estimate gives
    \begin{align*}
      \vert \lambda (f) - \lambda(I_{K,\Lambda_n} (f)) \vert \leq P_{\Lambda_n} (\lambda) \cdot \Vert f \Vert_K \xrightarrow{n \to \infty} 0. 
    \end{align*}

    \item \underline{$\beta \in \left( 0,1 \right)$:} Consider the monotonic and bounded sequence 
    \begin{align*}
      \left( P_{\Lambda_n}(\lambda) \right)_{n \in \mathbb{N}} \qquad \text{and its limit} \qquad C := \lim\limits_{n \to \infty} P_{\Lambda_n} (\lambda).
    \end{align*}
    For $C = 0$ we have pointwise convergence. Otherwise, we have $C > 0$ and
    \begin{eqnarray*}
    \lefteqn{
      \vert \lambda(f) - \lambda (I_{K,\Lambda_n} (f)) \vert^\beta
     } \\ 
     &&       
       = \vert \lambda(f) - \lambda (I_{K,\Lambda_n} (f)) \vert^\beta \cdot P_{\Lambda_n} (\lambda)^{1-\beta} \cdot P_{\Lambda_n}(\lambda)^{\beta - 1} 
      \\
      &&
      \leq \vert \lambda_{n+1} (f) - \lambda_{n+1} (I_{K,\Lambda_n} (f)) \vert^\beta \cdot P_{\Lambda_n} (\lambda_{n+1} )^{1-\beta} \cdot P_{\Lambda_n}(\lambda)^{\beta - 1} 
      \\
      && 
      \leq C^{\beta - 1} \cdot \Vert f \Vert_K^\beta \cdot P_{\Lambda_n}(\lambda_{n+1}) \xrightarrow{n \to \infty} 0.
    \end{eqnarray*}
    This also yields
    \begin{align*}
      \vert \lambda(f) - \lambda (I_{K,\Lambda_n} (f)) \vert \xrightarrow{n \to \infty} 0.
    \end{align*}
    
    \item \underline{$\beta \in \left[ 1, \infty \right)$:} Similarly, we can estimate
    \begin{eqnarray*}
    \lefteqn{
      \vert \lambda(f) - \lambda(I_{K,\Lambda_n}(f)) \vert^\beta 
      } \\
      &&
      \leq \vert \lambda_{n+1} (f) - \lambda_{n+1} (I_{K,\Lambda_n}(f)) \vert^\beta \cdot P_{\Lambda_n}(\lambda_{n+1} )^{1-\beta} \cdot P_{\Lambda_n}(\lambda)^{\beta - 1}
      \\
      && 
      \leq
      \Vert \lambda \Vert_K^{\beta - 1} \cdot \Vert f \Vert_K^\beta \cdot P_{\Lambda_n}(\lambda_{n+1}) \xrightarrow{n \to \infty} 0,
    \end{eqnarray*}
    where we used the bound $P_{\Lambda_n}(\lambda) \leq \| \lambda \|_K$ in the second inequality.

    \item \underline{$\beta = \infty$:} It can be shown that
    $$
      \lim\limits_{n \to \infty} \Vert I_{K,\Lambda_n} \Vert_{K}^2 
      = 
      \Vert I_{K,\Lambda_1}(f) \Vert_{K}^2 + \sum\limits_{i=1}^\infty \frac{\vert \lambda_{n+1}(f) - \lambda_{n+1}(I_{K,\Lambda_n}(f)) \vert^2}{P_{\Lambda_n}(\lambda_{n+1})^2} 
      \leq 
      \Vert f \Vert_{K}^2,
    $$
    see also \cite{wenzel2022adaptive}. This immediately implies the convergence
    \begin{align*}
      \frac{\vert \lambda_{n+1}(f) - \lambda_{n+1}(I_{K,\Lambda_{n}}(f)) \vert}{P_{\Lambda_{n}}(\lambda_{n+1})} \xrightarrow{n \to \infty} 0
      \qquad \mbox{ for } n \to \infty.
    \end{align*}
    As in the case before, we estimate
    \begin{align*}
      \vert \lambda(f) - \lambda(I_{K,\Lambda_n}(f)) \vert \leq \frac{\vert \lambda_{n+1} (f) - \lambda_{n+1} (I_{K,\Lambda_n}(f)) \vert}{P_{\Lambda_n}(\lambda_{n+1})} \cdot \Vert \lambda \Vert_K \xrightarrow{n \to \infty} 0.
    \end{align*}
  \end{itemize}
  Altogether, this proves the stated pointwise convergence for $\beta \in \left[ 0, \infty \right]$.
\end{proof}

\subsection{Geometric greedy algorithm}
The geometric greedy ($h$-greedy) algorithm (cf.~\cite{DeMarchi2005}) aims to generate an optimal sepa\-ration between the data points. It selects the functional that has the largest distance to the current data set, i.e.,
\begin{align} \label{hgreedy_select}
  \underset{\mu \in \Lambda_n}{\min} \; \Vert \lambda_{n+1} - \mu \Vert_{K} = 
  \underset{\lambda \in \Gamma}{\sup} \; \underset{\mu \in \Lambda_n}{\min} \; \Vert \lambda - \mu \Vert_{K} = h_{\Lambda_n, \Gamma}.
\end{align}

The convergence of this greedy algorithm for totally bounded supersets is a direct consequence of Theorem \ref{filldistance_convergence} and Lemma \ref{totalbounded_convergence}. 

\begin{Theorem} \label{hgreedy_convergence}
  Let $\Gamma \subset \mathcal{H}_K^*$ be totally bounded.
  Moreover, let $\left( \Lambda_n \right)_{n \in \mathbb{N}}$ be chosen via the geometric greedy algorithm (\ref{hgreedy_select}). 
  Then, we have $h_{\Lambda_n,\Gamma} \xrightarrow{n \to \infty} 0$ and
  \begin{align*}
    \Vert f - I_{K,\Lambda_n}(f) \Vert_{K} \xrightarrow{n \to \infty} 0 \qquad \text{for all } f \in \mathcal{H}_{K,\Gamma}.
  \end{align*}
\end{Theorem}

If $\Gamma$ is parameterized by another metric space $\Omega$, then the geometric greedy algorithm can also be performed in the parameter space. The convergence of this approach follows from Theorem \ref{parameter_convergence} and Lemma \ref{totalbounded_convergence}.

\begin{Remark}
  In practical cases, we will mostly deal with (large) finite supersets $\Gamma \subset \mathcal{H}_K^*$. It should be mentioned here that the discussed greedy methods are well-suited to thin out large data sets to reduce numerical redundancies (cf.~\cite{Dyn2002}). This can also be observed in the numerical tests from the next section.
\end{Remark}

\section{Image reconstruction from Radon data}
\label{sec:numerics}

In~\cite{DeMarchi2018}, a novel concept for kernel-based algebraic reconstruction from scattered Radon data was proposed. The corresponding approach uses a discretized version of the \textit{Radon transform} (see \cite[Chapter 9]{Iske2018Approx}) to reduce the initial operator equation to a generalized interpolation problem. For given radius $r \in \mathbb{R}$ and angle $\theta \in \left[ 0, \pi \right)$, the respective functionals were given by the line integral operators in~\eqref{radon_functional}, i.e.,
$$
    \mathcal{R}_{r,\theta} (f) = \int\limits_{\mathbb{R}} f( r \cdot \cos(\theta) - s \cdot \sin(\theta) , r \cdot \sin(\theta) + s \cdot \cos(\theta) ) \; ds.
$$
However, as shown in~\cite{DeMarchi2018}, this approach fails to work for the standard radially symmetric kernels, since the Radon functionals do not belong to their native dual space. A possible solution to this problem is by using weighted positive definite kernel functions of the form
\begin{align*}
  K_w(x,y) := w(x) \cdot K(x,y) \cdot w(y) \qquad \text{for } x,y \in \mathbb{R}^d, 
\end{align*}
where $w:\mathbb{R}^d \longrightarrow \mathbb{R}$ is an integrable weight function and $K$ is a standard kernel.

In this section, we discuss only a few supporting numerical experiments concerning kernel-based interpolation from scattered Radon data. Our purpose is to provide a proof of concept. To this end, we illustrate selected numerical effects, where our focus is on greedy data selection.
Adapting to the setup in~\cite{DeMarchi2018}, we consider working with the weighted kernel function
\begin{align*}
  K_w(x,y) = e^{- \beta \Vert x \Vert_2^2} \cdot e^{- \alpha \Vert x - y \Vert_2^2} \cdot e^{- \beta \Vert y \Vert_2^2} \qquad \text{for } x,y \in \mathbb{R}^d,
\end{align*}
for shape parameters $\alpha, \beta > 0$.
Note that the chosen kernel $K_w$ is {\em smooth}, and therefore does not comply with our initial idea to work with {\em rough} kernels. On the other hand, the choice of $K_w$ allows us to provide numerical results concerning the convergence of greedy methods, as discussed in Section~\ref{sec:greedy}.

For our numerical tests, we considered using the {\em Shepp-Logan phantom} (cf.~\cite{SheppLogan1974}) 
and the {\em smooth phantom} from~\cite[Section 6]{Rieder2003}, with smoothness parameter $\nu = 3$.
The two targets are visualised in the top left images of Figure \ref{fig:reconstructions_shepp_logan} and \ref{fig:reconstructions_smooth_phantom}, respectively. 
To create a set $\{ \lambda_i(f) \}_{\lambda_i \in \Lambda_N}$ of scattered Radon samples, we generated $N = 10,000$ pairs $(r_i,\theta_i)$ of random radii $r_i$ and angle parameters $\theta_i$ from the domain $\left[ - 1, 1 \right] \times \left[ 0, \pi \right)$, for $i=1,\ldots, N$, each corresponding to one Radon functional $\lambda_i$ of the form (\ref{radon_functional}). The formula for the Riesz representers and the inner products in the native dual space can be found in \cite[Section 4]{DeMarchi2018}. For the approximation quality, we recorded the {\em root mean squared error} (RMSE) on a $256 \times 256$ pixel grid
\begin{align*}
  \left( \frac{1}{256^2} \sum\limits_{1 \leq i,j \leq 256} (f(x_i,y_j) - I_{K_w,\Lambda_N}\left[f\right] (x_i,y_j))^2 \right)^{1/2}.
\end{align*}
We applied the algorithms $P$-greedy, geometric greedy, $f$-greedy, $f/P$-greedy and the psr-greedy for $M=2, 500$ iterations in order to reconstruct each phantom.

The visual results concerning the Shepp-Logan phantom are shown in Figure~\ref{fig:reconstructions_shepp_logan},
where we used the shape parameters $\alpha_{SL} = 2000$ and $\beta_{SL} = 2$.
While both the psr-greedy and the $f$-greedy reconstruction achieved competitive reconstructions, 
the reconstructions from the algorithms $P$-greedy, $f/P$-greedy and $h$-greedy are inferior. 

\begin{figure}[htb]
  \centering
  \includegraphics[scale = 0.28]{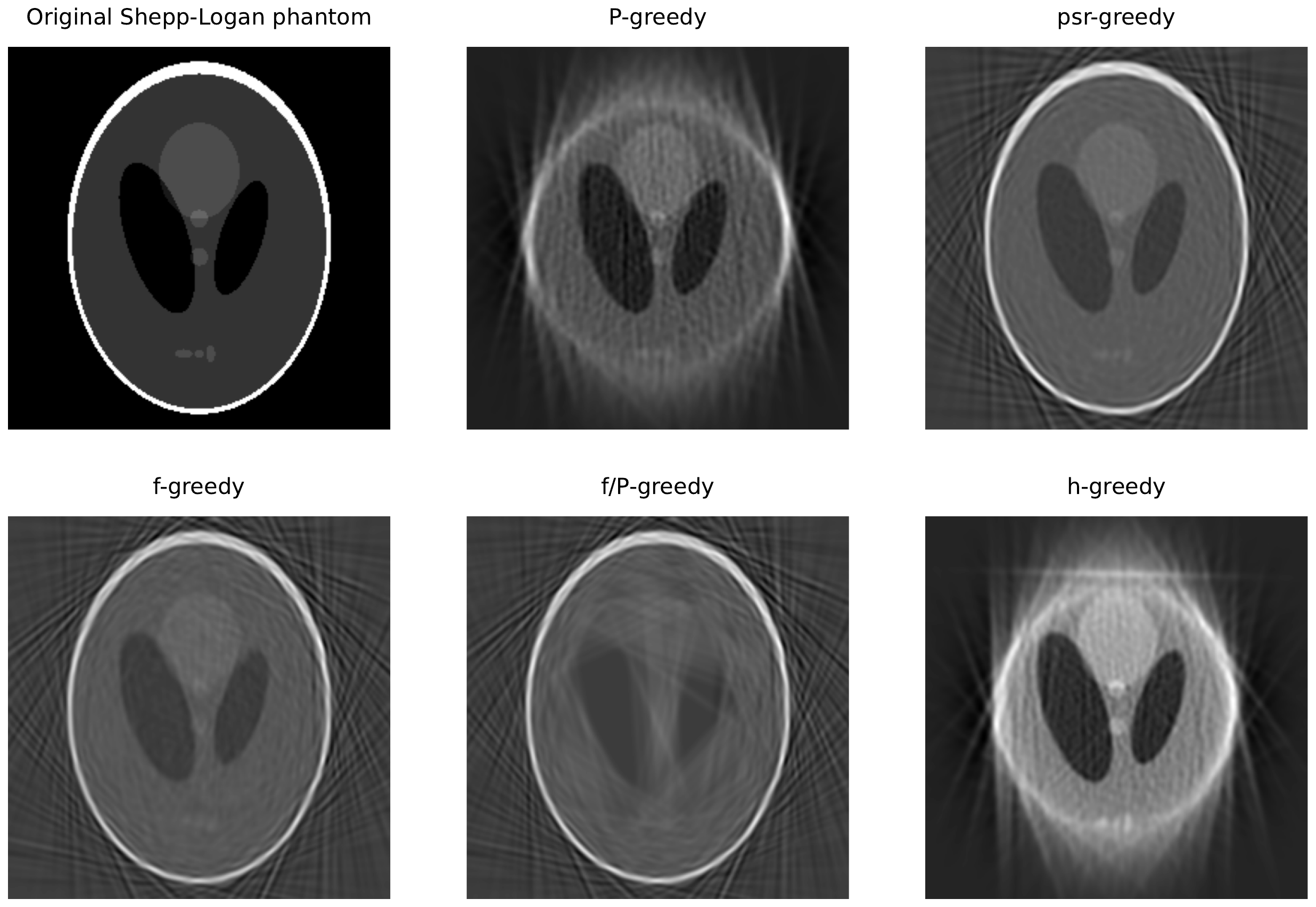}
  \caption{Reconstruction of Shepp-Logan phantom for different greedy methods after $M = 2,500$ iterations.}
  \label{fig:reconstructions_shepp_logan}
\end{figure}

Figure~\ref{fig:measurements_shepp_logan} shows the decay of the RMSE for the different greedy methods (left),
and the growth of the spectral condition number for their corresponding interpolation matrices (right).
Note that the target-dependent algorithms ($f$-greedy, $f/P$-greedy and the psr-greedy) achieved the smallest RMSE values, 
whereas the target-independent algorithms ($P$-greedy, geometric greedy) achieved the smallest condition numbers. 
We can conclude that the psr-greedy algorithm clearly achieved the best tradeoff between approximation quality and numerical stability.

\begin{figure}[htb]
  \centering
  \includegraphics[scale = 0.39]{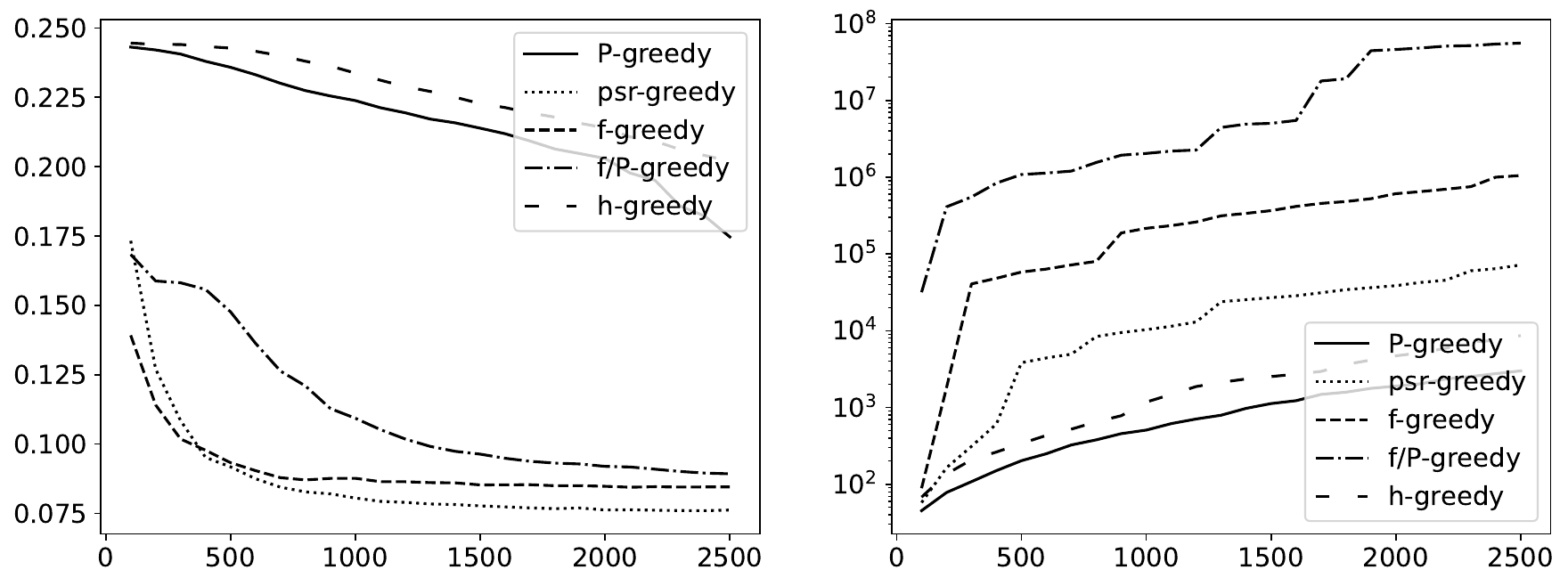}
  \caption{{\bf Reconstruction of the Shepp-Logan phantom.} 
  Decay of RMSE (left) and growth of spectral condition number $\kappa_2(A_{K,\Lambda_n})$ (right) 
  as a function of $n$, i.e., the number of the selected Radon functionals.}
  \label{fig:measurements_shepp_logan}
\end{figure}

Now let us discuss our numerical results concerning the smooth phantom in Figure~\ref{fig:measurements_smooth_phantom} (top left),
where we used the shape parameters $\alpha_{SP} = 700$ and $\beta_{SP} = 3$.
In this test case, we see that the target-dependent algorithms provide better reconstructions but worser condition numbers, 
in comparison to the target-independent algorithms. 

A comparison between the two test cases, Shepp-Logan and smooth phantom, shows that the reconstruction quality for the smooth phantom is clearly superior.
In fact, for the smooth phantom, all five greedy algorithms provided highly competitive reconstructions, see Figure~\ref{fig:reconstructions_smooth_phantom}. 

\begin{figure}[h!]
  \centering
  \includegraphics[scale = 0.28]{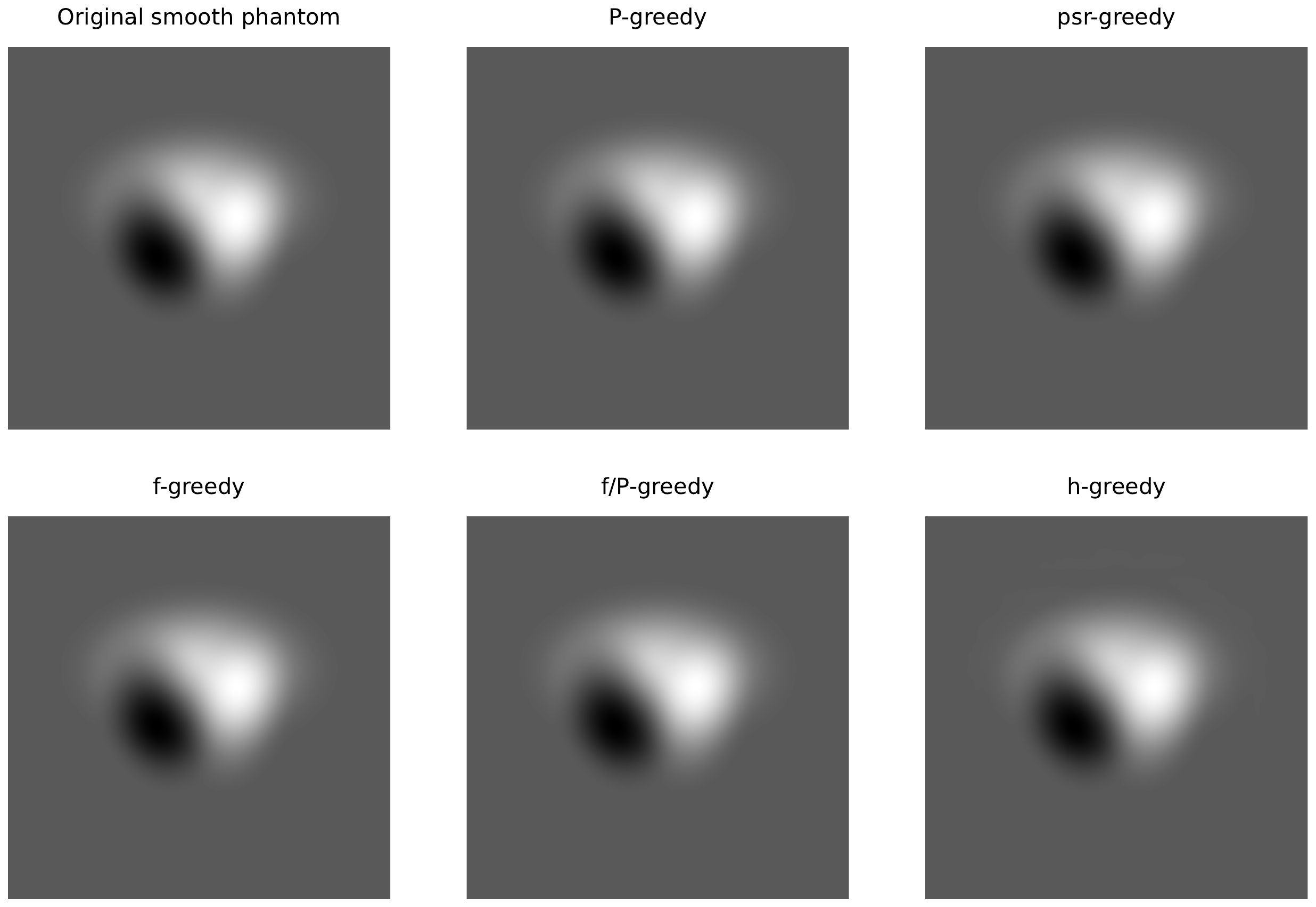}
  \caption{Reconstruction of smooth phantom with smothness parameter $\nu = 3$ for five different greedy methods after $M = 2,500$ iterations.}
  \label{fig:reconstructions_smooth_phantom}
\end{figure}

\begin{figure}[h!]
  \centering
  \includegraphics[scale = 0.39]{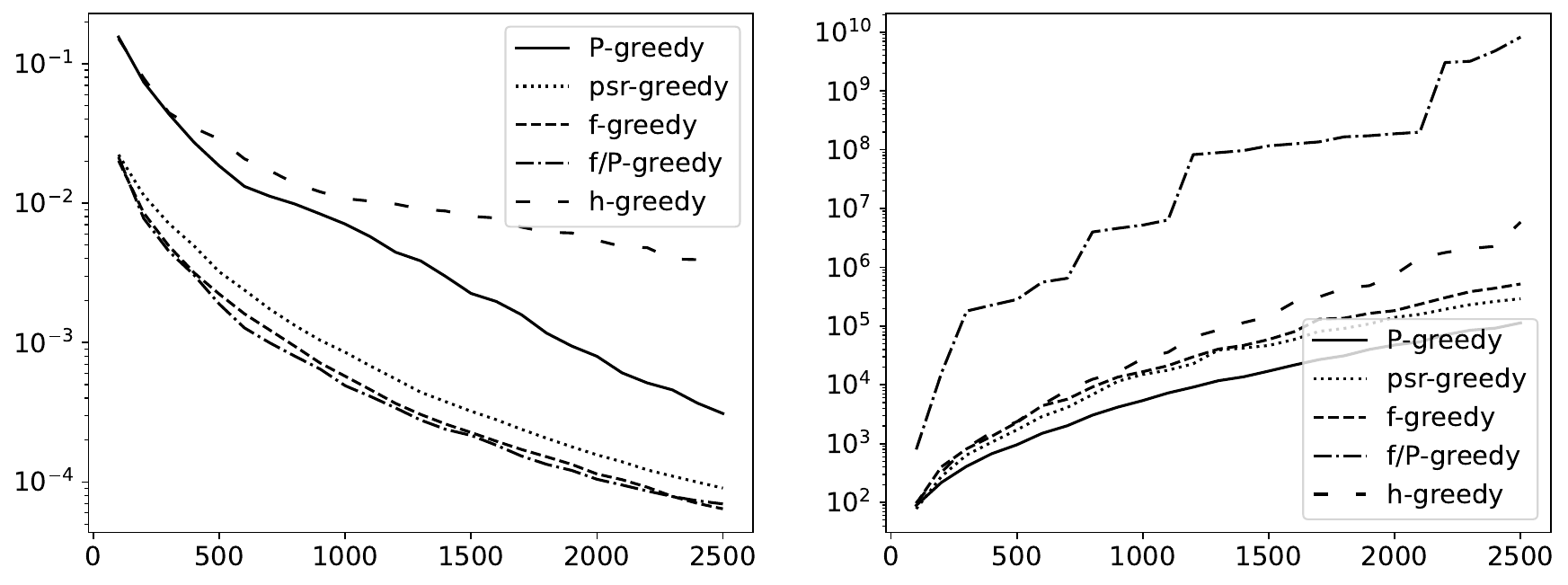}
  \caption{{\bf Reconstruction of the smooth phantom.} 
  Decay of RMSE (left) and growth of spectral condition number $\kappa_2(A_{K,\Lambda_n})$ (right) 
  as a function of $n$, i.e., the number of the selected Radon functionals.}
  \label{fig:measurements_smooth_phantom}
\end{figure}

\begin{table}[htb]
  \begin{minipage}{0.49\textwidth}
    \centering
    \begin{tabular}{lcc}
      \hline
       \textbf{Method}      &   RMSE &   $\kappa_2(A_{K,\Lambda_M})$ \\
      \hline
        no greedy &  8.30e-02 & 1.08e+09 \\
        $P$-greedy & 1.75e-01 & 3.00e+03 \\
        psr-greedy & 7.62e-02 & 7.20e+04 \\
        $f$-greedy & 8.45e-02 & 1.05e+06 \\
        $f/P$-greedy & 8.93e-02 & 5.58e+07 \\
        $h$-greedy & 2.02e-01 & 8.64e+03 \\
      \hline
      \end{tabular}
  \end{minipage}
  \begin{minipage}{0.48\textwidth}
    \centering
    \begin{tabular}{lcc}
      \hline
       \textbf{Method}      &   RMSE &   $\kappa_2(A_{K,\Lambda_M})$ \\
      \hline
        no greedy & 1.52e-05 & 3.32e+13 \\
        $P$-greedy & 3.10e-04 & 1.13e+05 \\
        psr-greedy & 9.05e-05 & 2.92e+05 \\
        $f$-greedy & 6.43e-05 & 5.20e+05 \\
        $f/P$-greedy & 6.97e-05 & 8.26e+09 \\
        $h$-greedy & 3.62e-03 & 5.87e+06 \\
      \hline
      \end{tabular}
  \end{minipage}
  \caption{RSMEs and condition numbers for Shepp-Logan (left) and smooth phantom (right).}
  \label{tbl:final_iteration}
\end{table}

Table~\ref{tbl:final_iteration} shows the RMSE values and condition numbers after the final iteration for each greedy algorithm and each phantom. 
Moreover, we recorded the RMSEs and condition numbers for the reconstruction from the set of all $N = 10,000$ Radon values. 
Note that some of the greedy algorithms achieved comparable RMSEs, while improving the condition number quite significantly, 
in comparison to the case, where all $N = 10,000$ Radon values were used (i.e., no greedy).


\section{Acknowledgment}
We acknowledge the support by the Deutsche Forschungsgemeinschaft (DFG) 
within the Research Training Group GRK 2583 ``Modeling, Simulation and Optimization of Fluid Dynamic Applications''. 
Moreover, we wish to thank Tizian Wenzel for several fruitful discussions.
Finally, we thank the anonymous reviewers for their useful comments and suggestions.



\begin{thebibliography}{10}

\providecommand{\url}[1]{\texttt{#1}}
\expandafter\ifx\csname urlstyle\endcsname\relax
  \providecommand{\doi}[1]{doi: #1}\else
  \providecommand{\doi}{doi: \begingroup \urlstyle{rm}\Url}\fi

\bibitem{Aboiyar2010}
T.~Aboiyar, E.H.~Georgoulis, and A.~Iske:
Adaptive {ADER} methods using kernel-based polyharmonic spline {WENO} reconstruction.
{\sl SIAM Journal on Scientific Computing}~{\bf 32} (6), 2010, 3251--3277.

\bibitem{Buhmann2003}
M.D.~Buhmann:
{\sl Radial Basis Functions}.\\
Cambridge University Press, Cambridge, UK, 2003.

\bibitem{DeMarchi2018}
S.~De Marchi, A.~Iske, and G.~Santin:
Image reconstruction from scattered Radon data by weighted positive definite kernel functions.
{\sl Calcolo}~{\bf 55}, 2018, 1--24.

\bibitem{DeMarchi2005}
S.~De Marchi, R.~Schaback, and H.~Wendland:
Near-optimal data-independent point locations for radial basis function interpolation.
{\sl Advances in Computational Mathematics}~{\bf 23}, 2005, 317--330.

\bibitem{DeVore2013}
R.~DeVore, G.~Petrova, and P.~Wojtaszczyk:
Greedy algorithms for reduced bases in {B}anach spaces.
{\sl Constructive Approximation}~{\bf 37} (3), 2013, 455--466.

\bibitem{Dutta2021}
S.~Dutta, M.W.~Farthing, E.~Perracchione, G.~Savant, and M.~Putti:
A greedy non-intrusive reduced order model for shallow water equations.
{\sl Journal of Computational Physics}~{\bf 439}, 2021, 110378.

\bibitem{Dyn2002}
N.~Dyn, M.S.~Floater, and A.~Iske:
Adaptive thinning for bivariate scattered data.
{\em Journal of Computational and Applied Mathematics}~{\bf 145} (2), 505--517.

\bibitem{Iske1995Hermite}
A.~Iske:
Reconstruction of functions from generalized {H}ermite-{B}irkhoff data.
{\sl Series in Approximations and Decompositions}~{\bf 6}, 1995, 257--264.

\bibitem{Iske2018Approx}
A.~Iske:
{\sl Approximation Theory and Algorithms for Data Analysis}.
Springer International Publishing, 2018.

\bibitem{Iske1996}
A.~Iske and T.~Sonar:
On the structure of function spaces in optimal recovery of point functionals for {ENO}-schemes by radial basis functions.
{\sl Numerische Mathematik}~{\bf 74} (2), 1996, 177--201.

\bibitem{Karvonen2022}
T.~Karvonen. 
Error bounds and the asymptotic setting in kernel-based approximation. 
{\sl Dolomites Res.\ Notes Approx.}~{\bf 15}, 65--77, 2022.

\bibitem{LeVeque2002}
R.L.~LeVeque.
{\sl Finite Volume Methods for Hyperbolic Problems}.
Cambridge University Press, Cambridge, UK, 2002.

\bibitem{Mueller2009Thesis}
S.~M\"uller:
{\sl Komplexit\"at und {S}tabilit\"at von kernbasierten {R}ekonstruktionsmethoden}.
Dissertation, Georg-August-Universit\"at G\"ottingen, 2009.

\bibitem{Mueller2009}
S.~M\"uller and R.~Schaback:
A {N}ewton basis for kernel spaces.
{\sl Journal of Approximation Theory}~{\bf 161} (2), 2009, 645--655.

\bibitem{Narcowich2006}
F.J.~Narcowich, J.D.~Ward, and H.~Wendland: 
Sobolev error estimates and a Bernstein inequality for scattered data interpolation via radial basis functions. 
{\sl Constr.\ Approx.}~{\bf 24}(2), 175--186, 2006.

\bibitem{Natterer2001} 
F.~Natterer: 
{\sl The Mathematics of Computerized Tomography}. 
Classics in Applied Mathematics, vol.~32. SIAM, Philadelphia, 2001.

\bibitem{Pazouki2011}
M.~Pazouki and R.~Schaback:
Bases for kernel-based spaces.
{\sl Journal of Computational and Applied Mathematics}~{\bf 236} (4), 2011, 575--588.

\bibitem{Rieder2003}
A.~Rieder and A.~Faridani:
The {S}emidiscrete {F}iltered {B}ackprojection {A}lgorithm {I}s {O}ptimal for {T}omographic {I}nversion.
{\sl SIAM Journal on Numerical Analysis}~{\bf 41} (3), 2003, 869--892.

\bibitem{Schaback_greedy2019}
R.~Schaback:
A greedy method for solving classes of PDE problems.
arXiv Preprint 1903.11536, {\tt https://doi.org/10.48550/arXiv.1903.11536}, 2019.

\bibitem{Schaback2000}
R.~Schaback and H.~Wendland:
Adaptive greedy techniques for approximate solution of large RBF systems.
{\sl Numerical Algorithms}~{\bf 24} (3), 2000, 239--254.

\bibitem{Schaback2006}
R.~Schaback and J.~Werner:
Linearly constrained reconstruction of functions by kernels with applications to machine learning.
{\sl Advances in Computational Mathematics}~{\bf 25}, 2006, 237--258.

\bibitem{SheppLogan1974}
L.A.~Shepp and B.F.~Logan:
The {F}ourier reconstruction of a head section.
{\sl IEEE Transactions on Nuclear Science}~{\bf 21} (3), 1974, 21--43.

\bibitem{Wendland2005}
H.~Wendland:
{\sl Scattered Data Approximation}.
Cambridge University Press, Cambridge, UK, 2005.

\bibitem{Wenzel2023standard}
T.~Wenzel, G.~Santin, and B.~Haasdonk:
Analysis of target data-dependent greedy kernel algorithms: convergence rates for $f$-, $f \cdot P$- and $f/P$-greedy.
{\sl Constructive Approximation}~{\bf 57} (1), 2023, 45--74.

\bibitem{wenzel2022adaptive}
T.~Wenzel, D.~Winkle, G.~Santin, and B.~Haasdonk:
Adaptive meshfree solution of linear partial differential equations with {PDE}-greedy kernel me\-thods.
arXiv preprint, arXiv:2207.13971, 2022.

\bibitem{Wu1992Hermite}
Z.~Wu:
Hermite-Birkhoff interpolation of scattered data by radial basis functions.
{\sl Approximation Theory Appl.}~{\bf 8}, 1992, 1--10.

\end{thebibliography}
\end{document}